\documentclass[12pt,reqno]{amsart}
\usepackage{amsmath}
\usepackage{amssymb,latexsym}
\usepackage[matrix,arrow]{xy}
\usepackage[usenames]{color}
\usepackage{slashed}
\usepackage{parskip}
\usepackage{amsmath, amssymb, graphics}
\usepackage{graphicx}

\newcommand{\mathsym}[1]{{}}

\newtheorem{thm}{Theorem}[section]
\newtheorem{cor}[thm]{Corollary}
\newtheorem{lem}[thm]{Lemma}

\theoremstyle{definition}

\numberwithin{equation}{section}
\theoremstyle{remark}
\newtheorem{rem}{Remark}[section]
\theoremstyle{example}

\theoremstyle{conjecture}
\newtheorem{con}{Conjecture}[section]
\newcommand{\be}{\begin{equation}}
\newcommand{\ee}{\end{equation}}
\newcommand{\ba}{\aligned}
\newcommand{\ea}{\endaligned}
\newcommand{\n}{\nabla}

\newcommand{\ld}{\lambda}

\newcommand{\into}{\int_{\Omega}}
\newcommand{\al}{\alpha}
\newcommand{\mr}{\mathrm}
\newcommand{\lf}{\left}
\newcommand{\rt}{\right}

\begin{document}
\title[estimates of the gaps between consecutive eigenvalues of Laplacian]{  estimates of the gaps between consecutive eigenvalues of Laplacian }
\author [Tao Zheng]{ Daguang Chen$^*$, Tao Zheng,  Hongcang Yang$^{**}$}

\thanks {}
\thanks{*The work of the first named author was partially supported by NSFC grant No. 11101234.}
 \thanks{**The work of the third named author was partially supported by NSFC and SF of CAS.}

\begin{abstract}
By the calculation of the gap of the consecutive eigenvalues of $\Bbb S^n$ with standard metric, using the Weyl's asymptotic formula, we know the order of the upper bound of this gap is $k^{\frac{1}{n}}.$
We conjecture that this order is also right for general Dirichlet problem of the Laplace operator, which is optimal if this conjecture holds, obviously.
In this paper, using new method, we solve this conjecture in the Euclidean space case intrinsically. We think our method is valid for
the case of general Riemannian manifolds and give some examples directly.
\end{abstract}
\maketitle
\renewcommand{\sectionmark}[1]{}
\section{introduction}
Let
$\Omega$
be a bounded domain in an
$n$-dimensional
complete Riemannian manifold
$M$
with boundary (possible empty).
Then the Dirichlet eigenvalue problem of Laplacian on
$
\Omega
$
is given by
\begin{equation}\label{DLRM}
\left\{
\aligned
\Delta u=&-\ld u,\;\mbox{in}\;\Omega,\\
u=&0,\;\;\;\;\;\;\;\;\mbox{on}\; \partial \Omega,
\endaligned
\right.
\end{equation}
where
$
\Delta
$
is Laplaican on $M$.
It is well known that the spectrum of (\ref{DLRM}) has the real and purely discrete eigenvalues
\begin{equation}\label{dandiao}
0<\ld_1< \ld_2\leq \ld_3\leq \cdots\nearrow\infty,
\end{equation}
where each
$
\ld_i
$
has finite multiplicity which is repeated
according to its multiplicity. The corresponding orthonormal basis of
real eigenfunctions will be denoted
$
\{u_j\}_{j=1}^\infty
$.
We go forward under the assumption that
$
L^2(\Omega)
$
represents the real Hilbert space of
real-valued
$
L^2
$
functions on
$\Omega$.
We put
$
\ld_{0}=0
$
if
$
\partial \Omega= \emptyset
$.

An important aspect of estimating higher eigenvalues is to obtain as precise as possible the estimate of gaps of  consecutive eigenvalues of (\ref{DLRM}).
In this regard, we will review some important results on the estimates of  eigenvalue problem (\ref{DLRM}).

For the upper bound of the gap of consecutive eigenvalues of (\ref{DLRM}), when $\Omega$
is a bounded domain in an
$2$-dimensional
Euclidean space
$
\Bbb{R}^2,
$
in 1956, Payne, P\'{o}lya and Weinberger (cf.\cite{PPW55}
and \cite{PPW56}) proved
\begin{equation}\label{ppw1}
\ld_{k+1}-\ld_{k}\leq\frac{2}{k}\sum_{i=1}^{k}\ld_{i}.
\end{equation}

C. J. Thompson \cite{thompson1969}, in 1969, extended (\ref{ppw1}) to $n$-dimensional case and obtained
\begin{equation}\label{ppwtype}
\ld_{k+1}-\ld_{k}\leq\frac{4}{nk}\sum_{i=1}^{k}\ld_{i}.
\end{equation}
Hile and Protter \cite{HP} improved
(\ref{ppwtype}) to
\begin{equation}\label{hp1}
\sum_{i=1}^{k}\frac{\ld_{i}}{\ld_{k+1}-\ld_{i}}\geq\frac{nk}{4}.
\end{equation}
Yang (cf. \cite{Yang91} and more recently \cite{CY4}) has obtained a sharp inequality
\begin{equation}\label{CY41}
\sum_{i=1}^{k}(\ld_{k+1}-\ld_{i})\left(\ld_{k+1}-\left(1+\frac{4}{n}\right)\ld_{i}\right)\leq0.
\end{equation}
From
(\ref{CY41}),
one can infer
\begin{equation}\label{CY42}
\ld_{k+1}\leq \frac{1}{k}\left(1+\frac{4}{n}\right)\sum_{i=1}^{k}\ld_{i}.
\end{equation}
The inequalities
(\ref{CY41})
and
(\ref{CY42})
are called Yang's first inequality and second inequality, respectively (cf. \cite{Ash98, Ash02, ab1996, harrellstubbe}).
Also we note that Ashbaugh and Benguria gave an optimal estimate for
$k=1$ (cf. \cite{ab1991,ab1992, ab19921}).
From the Chebyshev's
inequality, it is easy to prove the following relations
$$
(\ref{CY41})\Longrightarrow(\ref{CY42})\Longrightarrow(\ref{hp1})\Longrightarrow(\ref{ppwtype}).
$$
From (\ref{CY41}), Cheng and Yang \cite{CY2005} obtained
\begin{equation}\label{Yangdiff1}
\ld_{k+1}-\ld_{k}\leq  2
\left[\left(\frac{2}{n}\frac{1}{k}\sum\limits_{i=1}^{k}\ld_{i}\right)^{2}-
\left(1+\frac{4}{n}\right)
\frac{1}{k}\sum\limits_{i=1}^{k}\left(\ld_{i}-\frac{1}{k}\sum\limits_{j=1}^{k}\ld_{j}\right)^{2}\right]^{\frac{1}{2}}.
\end{equation}

Cheng and Yang \cite{CY4}, using their recursive formula, obtained
\begin{equation}\label{diffyong}
\ld_{k+1}\leq C_{0}(n)k^{\frac{2}{n}}\ld_{1},
\end{equation}
where $C_{0}(n)\leq 1+\frac{4}{n}$ is a constant (see Cheng and
Yang's paper \cite{CY4}). From the Weyl's asymptotic formula (cf.
\cite{weyl}), we know that the upper bound (\ref{diffyong}) of Cheng
and Yang  is best possible in the meaning of the order on $k$.

For a complete Riemannian manifold $M$, from the Nash's theorem \cite{Nash},  there exists an isometric
immersion
$$
\psi\,:\;M\longrightarrow \Bbb{R}^{N},
$$
where $\Bbb{R}^{N}$ is Euclidean space. The mean curvature of the immersion $\psi$ is denoted by $H$ and $|H|$ denotes its norm. Define
$$
\Phi=\{\psi\;|\psi\; \mbox{is an isometric immersion from $M$ into Euclidean space}\}.
$$

 When $\Omega$ is a bounded domain of a complete Riemannian manifold
$M$, isometrically immersed into a Euclidean space $\Bbb{R}^{N}$, Cheng and the first author \cite{ChenCheng} (cf. \cite{elsoufi,harrellmichel2007})obtained
\begin{equation}\label{chenchengwaiyun}
\sum_{i=1}^{k}(\ld_{k+1}-\ld_{i})^2\leq\frac{4}{n}\sum_{i=1}^{k}(\ld_{k+1}-\ld_{i})\left(\ld_{i}+\frac{n^2}{4}H_{0}^2\right),
\end{equation}
where
\begin{equation}\label{jinruchangshu}
H_{0}^{2}= \inf_{\psi\in\Phi} \sup_{\Omega}|H|^{2}.
\end{equation}
Using the recursive formula in Cheng and Yang \cite{CY4}, Cheng and the first author in \cite{ChenCheng} also deduced
\begin{equation}\label{chenchenganhan}
\ld_{k+1}+\frac{n^2}{4}H_{0}^2\leq C_{0}(n)k^{\frac{2}{n}}\left(\ld_{1}+\frac{n^2}{4}H_{0}^2\right),
\end{equation}
where
$H_{0}^{2}, C_{0}(n)$
are  given by (\ref{jinruchangshu}) and (\ref{diffyong}) respectively.

From (\ref{chenchengwaiyun}), we can get the gaps of the consecutive eigenvalues of Laplacian
\begin{equation}\label{chenchengwaiyundiff}
\ld_{k+1}-\ld_{k}\leq  2
\left(\left(\frac{2}{n}\frac{1}{k}\sum\limits_{i=1}^{k}\ld_{i}+\frac{n}{
2}H_{0}^2\right)^{2}- \left(1+\frac{4}{n}\right)
\frac{1}{k}\sum\limits_{i=1}^{k}\left(\ld_{i}-\frac{1}{k}\sum\limits_{j=1}^{k}\ld_{j}\right)^{2}\right)^{\frac{1}{2}}.
\end{equation}
\begin{rem}
When
$
\Omega
$
is an
$n$-dimensional
compact homogeneous Riemannian manifold, a compact minimal submanifold without boundary and a connected bounded domain in the standard unit sphere
$
\mathbb{S}^{N}(1)
$,
and  a connected bounded domain and a compact complex hypersurface without boundary of the  complex projective space
$\mathbb {CP}^{n}(4)$
with holomorphic sectional curvature 4,
many mathematicians have sutudied the universal inequalities for eigenvalues and the difference of the consecutive eigenvalues
(cf. \cite{CY2005,CY2006,CY2009,harrell,harrellmichel,harrellstubbe,PLi,YangYau,Leung,sun2008}).
\end{rem}
\begin{rem}
Another  problem is  the
lower bound of the gap of the first two eigenvalues. In general, there exists  the famous fundamental gap conjecture for the Dirichlet eigenvalue problem of the Schr\"{o}dinger
operator(cf.\cite{berg1983,ab1989,yau1986,wyau,yuzhong}and the references therein). The fundamental gap conjecture was solved by B. Andrews and J. Clutterbuck in \cite{AC11} .
\end{rem}

From (\ref{Yangdiff1}) and (\ref{chenchengwaiyundiff}), it is not difficult to see that both Yang's
estimate for the gap of  consecutive eigenvalues of (\ref{DLRM}) implicited in \cite{Yang91} and the estimate
from \cite{ChenCheng} are on the order of $k^{\frac{3}{2n}}$. However, by  the calculation of the gap of the consecutive eigenvalues of $\Bbb S^n$ with standard metric, using the Weyl's asymptotic formula, we know the order of the upper bound of this gap is $k^{\frac{1}{n}}.$ Hence we conjecture that
\begin{con}
Let $\Omega$ be a bounded domain in an $n$-dimensional complete Riemannian manifold $M$.For the Dirichlet problem (\ref{DLRM}), the upper bound for the gap of consecutive eigenvalues of Laplacian should be
\begin{equation}\label{Conj}
\ld_{k+1}-\ld_{k}\leq C_{n,\Omega}k^{\frac{1}{n}},
\end{equation}
where $C_{n,\Omega}$ is a constant dependent on  $\Omega$ itself and the dimension $n$.
\end{con}
\begin{rem}
The famous Panye-P\'{o}lya-Weinberger conjecture(cf.\cite{PPW55,PPW56,thompson1969,ab93,ab932})claims that, when $M=\mathbb{R}^n$, for Dirichlet eigenvalue problem (\ref{DLRM}), one should have
\be\label{ppwcon3}
\dfrac{\lambda_{k+1}}{\lambda_k}\leq\left.\dfrac{\lambda_2}{\lambda_1}\right|_{\Bbb B^n}= \left(\dfrac{j_{n/2,1}}{j_{n/2-1,1}}\right)^2,
\ee
where $\mathbb{B}^{n}$ is the $n$-dimensional unit ball in $\Bbb{R}^n,$ and
$j_{p,k}$ is the $k^{th}$ positive zero  of the Bessel function
$J_p(t)$. From the Weyl's asymptotic formula and (\ref{ppwcon3}), we know that the order of the upper bounder of the consective eigenvalues of eigenvalue problem (\ref{DLRM}) is just $k^{\frac{2}{n}}$. This is why we make this conjecture.
\end{rem}

In the following, the constant $C_{n,\Omega}$ are allowed to be different in different cases.

When $\Omega$ is a bouded domain in $\Bbb{R}^n$, for the Dirichlet eigenvalue problem
(\ref{DLRM}), we give the affirmative answer to  the  conjecture (\ref{Conj}).
\begin{thm}\label{MainThm1}
Let $\Omega \subset \Bbb R^n$ be a bounded domain in Euclidean space $\Bbb R^n$  and $\ld_{k}$ be the
$k^{th}$ eigenvalue of  the Dirichlet eigenvalue problem
(\ref{DLRM}). Then we have
\begin{equation}\label{U-inequ1}
\ld_{k+1}-\ld_{k}\leq C_{n,\Omega}k^{\frac 1n},
\end{equation}
where $C_{n,\Omega}= 4\ld_1\sqrt{\frac{C_{0}(n)}{n}}$, $C_{0}(n)$ is given by $(\ref{diffyong})$.
\end{thm}

We think our method of proving Theorem \ref{MainThm1} will be valid for the case of $n$-dimensional complete Riemannian manifold. Here, we give some examples directly.

\begin{cor}\label{gapexample1}
Let $\Omega \subset \Bbb H^{n}(-1) $ be a bounded domain in hyperbolic space
$\Bbb H^{n}(-1)$,  and $\ld_{k}$ be the
$k^{th}$ eigenvalue of  the Dirichlet eigenvalue problem
(\ref{DLRM}). Then we have
\begin{equation}\label{U-inequ3}
\ld_{k+1}-\ld_{k}\leq C_{n,\Omega}k^{\frac{1}{n}},
\end{equation}
where $C_{n,\Omega}$ depends on $\Omega$ and the dimension $n$, given by
\begin{equation}\label{Const1}
  C_{n,\Omega}=4\left[C_{0}(n)\left(\ld_{1}-\frac{(n-1)^2}{4}\right)\Big(\ld_{1}
+\frac{n^2}{4}H_{0}^2\Big)\right]^{\frac{1}{2}},
\end{equation}
$C_{0}(n) $ and  $H_{0}^2$ are the same as the ones  in (\ref{chenchenganhan}).
\end{cor}
In fact, by the comparison theorem for the distance function in Riemannian manifold, we have
\begin{cor}\label{MainThm2}
Let $\Omega\subset M$ be  a bounded domain of  an $n$-dimensional ($n\geq3$) simply connected complete noncompact Riemannian manifold $M$ with
sectional curvature $Sec$  satisfying
$$
 -a^2\leq Sec\leq -b^2,
$$
where $a\geq b\geq0$ are constants. Let $\ld_{k}$ be the $k^{th}$
eigenvalue of  the eigenvalue problem (\ref{DLRM}). Then we have
\begin{equation}\label{U-inequ2}
\ld_{k+1}-\ld_{k}\leq C_{n,\Omega}
k^{\frac{1}{n}}
\end{equation}
where $C_{n,\Omega}$ depends on $\Omega$ and
the dimension $n$, given by
\begin{equation}\label{Const2}
  C_{n,\Omega}=4\left[C_{0}(n)\left(\ld_{1}-\frac{(n-1)^2}{4}b^2+\frac{a^2-b^2}{4}\right)
\left(\ld_{1}+\frac{n^2}{4}H_{0}^2\right)\right]^{\frac12},
\end{equation}
$C_{0}(n)$ and $H_{0}^2$ are the same as the ones in (\ref{chenchenganhan}).
\end{cor}
\begin{rem}
 Under the same assumption of Corollary \ref{MainThm2}, Lu and the first two authors \cite{CZL} obtained
\begin{equation*}
\sum_{i=1}^{k}(\lambda_{k+1}-\lambda_{i})^2
\leq  4\sum_{i=1}^{k}(\lambda_{k+1}-\lambda_{i})\left(\ld_{i}-\frac{(n-1)^2}{4}b^2+\frac{n-1}{2}(a^2-b^2)\right).
\end{equation*}
Therefore, one can get
\begin{equation*}
\ba
  \ld_{k+1}-\ld_{k}\leq 2&\left[ \left(\frac{2}{k}\sum_{i=1}^{k}\ld_i-\frac{(n-1)^2}{4}b^2+\frac{n-1}{2}(a^2-b^2)\right)^2\right.\\
  &\left.\quad-\frac 5{k}\sum_{i=1}^{k}\left(\ld_i-
  \frac{1}{k}\sum_{i=1}^{k}\ld_j\right)^2  \right]^{\frac12}.
\ea\end{equation*}
\end{rem}

\section{  proofs of main results}
In this section, we will give the proof of Theorem \ref{MainThm1}.
In order to prove our main results, we need the following key lemma
\begin{lem}\label{gapkeylemma}
For the Dirichlet eigenvalue problem (\ref{DLRM}),  let $u_{k}$ be the orthonormal eigenfunction corresponding to the $k^{th}$ eigenvalue $\lambda_{k} $,
i.e.
\begin{equation*}
\left\{
\aligned
\Delta u_k=&-\ld_{k} u_k,\; \text{in} \; \Omega,\\
u_k=&0, \;\;\;\;\;\;\;\; \;\;\;\text{on}\;\partial \Omega,\\
\into u_{i}u_{j} & =  \delta_{ij}.
\endaligned
\right.
\end{equation*}

Then for any complex value function
$
g\in C^{3}(\Omega)\cap C^{2}(\overline{\Omega})
$
and
$
k,\,i\in \Bbb Z^{+},\, k>i\geq1,
$
we have
\begin{equation}\label{mainformula}
\aligned
\Big((\ld_{k+1}-\ld_{i})+(\ld_{k+2}-\ld_{i})\Big)\into |\nabla g|^2u_{i}^2
\leq &\into\Big|2\nabla g\cdot\nabla u_{i}+ u_{i}\Delta g \Big|^2\\
&+(\ld_{k+1}-\ld_{i})(\ld_{k+2}-\ld_{i})\into |gu_{i}|^2.
\endaligned
\end{equation}
\end{lem}
\begin{proof}
For $i<k$, define
$$
\left\{
\aligned
a_{ij} =& \into g u_{i}u_{j},   \\
b_{ij}=&\into \left(\nabla g\cdot\nabla u_{i}+\frac{1}{2}u_{i}\Delta g \right)u_{j},\\
\varphi_{i}=&gu_{i}-\sum_{j=1}^{k}a_{ij}u_{j},
\endaligned
\right.
$$
where
$
\nabla
$
denotes the gradient operator.
Obviously,
\begin{equation}\label{22222}
a_{ij}=a_{ji},\;\;\into \varphi_{i}u_{j}=0,\;\;\mbox{for}\;\; j=1,2,\cdots,k.
\end{equation}
Then, from the Stokes' theorem, we get
\begin{equation*}\label{ }
\ld_{j}a_{ij}=\into gu_{i}(-\Delta u_{j})=-\into (u_{i}\Delta g+g\Delta u_{i}+2\nabla g\cdot\nabla u_{i})u_{j},
\end{equation*}
i.e.
\begin{equation}\label{33333}
2b_{ij}= (\ld_{i}-\ld_{j})a_{ij}.
\end{equation}
From the Stokes' theorem, we have
\begin{equation*}\label{}
-2\into gu_{i}\n \overline{g}\cdot\n u_{i}=-\into g\n \overline{g}\cdot\n u_{i}^2
=\into (\n g\cdot\n \overline{g}+g\Delta \overline{g})u_{i}^2.
\end{equation*}
By the definition of $a_{ij}$, $b_{ij}$ and (\ref{33333}), we obtain
\begin{equation*}\label{}
\into |\n g|^2u_{i}^2=-2\into gu_{i}\left(\n \overline{g}\cdot \n u_{i}+\frac{1}{2}u_{i}\Delta \overline{g}\right)
=-2\sum_{j=1}^{\infty}a_{ij}\overline{b_{ij}}=\sum_{j=1}^{\infty}(\ld_{j}-\ld_{i})|a_{ij}|^2.
\end{equation*}
Similarly, from the Stokes' theorem,
(\ref{22222})
and
(\ref{33333}),  we have
\begin{equation}\label{11111}
\aligned
\into |\nabla \varphi_{i}|^2=&\into \overline{\varphi_{i}}\left(-\Delta \varphi_{i}\right)\\
=&\into \overline{\varphi_{i}}\left(-\Delta (gu_{i})+ \sum\limits_{j=1}^{k}a_{ij}\Delta u_{j} \right)\\
=&\into \overline{\varphi_{i}}\left(-2\nabla g\cdot \nabla u_{i}-\Delta g u_{i}-g\Delta u_{i} \right)\\
=&-\into \overline{\varphi_{i}}\left(2\nabla g\cdot \nabla u_{i}+\Delta g u_{i}-\ld_ig u_{i} \right)\\
=&-2\sum\limits_{j=k+1}^{\infty}\overline{a_{ij}}b_{ij}+\ld_{i}\sum\limits_{j=k+1}^{\infty}|a_{ij}|^2\\
=&\sum\limits_{j=k+1}^{\infty}(\ld_{j}-\ld_{i})|a_{ij}|^2+\ld_{i}\sum\limits_{j=k+1}^{\infty}|a_{ij}|^2.
\endaligned
\end{equation}

From the Rayleigh-Ritz inequality (cf. \cite{chavel}) and
(\ref{11111}), we have
\begin{equation*}\label{ }
\ld_{k+1}\leq\frac{\into |\nabla \varphi_{i}|^2}{\into |\varphi_{i}|^2}
=\frac{\sum_{j=k+1}^{\infty}(\ld_{j}-\ld_{i})|a_{ij}|^2}{\sum_{j=k+1}^{\infty}|a_{ij}|^2}+\ld_{i},
\end{equation*}
i.e.
\begin{equation}\label{fujidajixiao}
(\ld_{k+1}-\ld_{i})\sum_{j=k+1}^{\infty}|a_{ij}|^2\leq  \sum_{j=k+1}^{\infty}(\ld_{j}-\ld_{i})|a_{ij}|^2 .
\end{equation}
From the  Cauchy-Schwarz inequality, we have
\begin{equation*}\label{ }
\left(\sum_{j=k+1}^{\infty}(\ld_{j}-\ld_{i})|a_{ij}|^2\right)^2 \leq
\sum_{j=k+1}^{\infty}(\ld_{j}-\ld_{i})^2|a_{ij}|^2
\sum_{j=k+1}^{\infty} |a_{ij}|^2,
\end{equation*}
i.e.
\begin{equation}\label{fuschwarzinequality}
\begin{aligned}
&\left( \into |\nabla g|^2u_{i}^2-\sum_{j=1}^{k}(\ld_{j}-\ld_{i})|a_{ij}|^2\right)^2\\
\leq&\left(\into
|gu_{i}|^2-\sum_{j=1}^{k}|a_{ij}|^2\right)\left(\into\left|2\nabla
g\cdot\nabla u_{i}+u_{i}\Delta g \right|^2 -\sum_{j=
1}^{k}(\ld_{j}-\ld_{i})^2|a_{ij}|^2\right).
\end{aligned}
\end{equation}
Define
$$
\left\{
\aligned
\widetilde{B}(i)= &\into |gu_{i}|^2-\sum_{j=1}^{k}|a_{ij}|^2 =\sum_{j=k+1}^{\infty}|a_{ij}|^2>0,\\
\widetilde{A}(i)=&\into\left|2\nabla g\cdot\nabla u_{i}+ u_{i}\Delta g \right|^2-\sum_{j=1}^{k}(\ld_{j}-\ld_{i})^2|a_{ij}|^2
=\sum_{j=k+1}^{\infty}(\ld_{j}-\ld_{i})^2|a_{ij}|^2\geq0,\\
\widetilde{C}(i)=&\into |\n g|^2u_{i}^2-\sum_{j=1}^{k}(\ld_{j}-\ld_{i})|a_{ij}|^2
=\sum_{j=k+1}^{\infty}(\ld_{j}-\ld_{i})|a_{ij}|^2.
\endaligned
\right.
$$
Next, we will deduce the maxima of $\widetilde{C}(i)$ by using the Lagrange method of multipliers (cf.\cite{berger}).
For  any sequence $\{\phi_{ij}\}_{j=k+1}^\infty$ satisfying
\begin{equation*}
  \sum_{j=k+1}^\infty|\phi_{ij}|^2\leq \infty
\end{equation*}
we define the function,
\begin{equation*}\label{}
\aligned
\Phi(|\phi_{ij}|,\mu,\ld)
 =&\sum_{j=k+1}^{\infty}(\ld_{j}-\ld_{i})|\phi_{ij}|^2
 +\mu\left(\sum_{j=k+1}^{\infty}(\ld_{j}-\ld_{i})^2|\phi_{ij}|^2-\widetilde{A}(i)\right)\\
 &+\ld\left(\sum_{j=k+1}^{\infty}|\phi_{ij}|^2-\widetilde{B}(i)\right).
\endaligned
\end{equation*}
where $\mu, \ld$ are two real parameters.

Assume $\{a_{ij}\}_{j=k+1}^{\infty}$ is the extreme point of $\varphi$. Then for any $\{\psi_{ij}\}_{j=k+1}^{\infty}$
satisfying $\sum\limits_{j=k+1}^{\infty}|\psi_{ij}|^2<\infty$, from
$$
\left.\frac{ \mathrm{d}}{\mathrm{d}t }\right|_{t=0} \Phi(|a_{ij}|+t|\psi_{ij}|)=0
$$
we have
\be\label{jizhidian}
2\sum\limits_{j=k+1}^{\infty}|a_{ij}||\psi_{ij}|\left((\ld_{j}-\ld_{i})+\mu (\ld_{j}-\ld_{i})^2+\ld\right)=0.
\ee
Taking
$$
\psi_{ij}=
\left\{
\ba
1,&\;\;j=p,\\
0,&\;\;\mbox{otherwise},
\ea
\right.
$$
in (\ref{jizhidian}), we have
\be\label{jizhidianyong}
|a_{ip}|\left((\ld_{p}-\ld_{i})+\mu (\ld_{p}-\ld_{i})^2+\ld\right)=0,\;p=k+1,\cdots.
\ee
From
\begin{equation*}
\left\{
\begin{aligned}
&\frac{\partial \Phi}{\partial \mu}=0,\\
&\frac{\partial \Phi}{\partial \ld}=0,
\end{aligned}
\right.
\end{equation*}
we have the two constraint conditions
\begin{equation}\label{jizhidianyong2}
\left\{\begin{aligned}
&\sum_{j=k+1}^{\infty}(\ld_{j}-\ld_{i})^2|a_{ij}|^2=\widetilde{A}(i),\\
&\sum_{j=k+1}^{\infty}|a_{ij}|^2=\widetilde{B}(i).
\end{aligned}\right.
\end{equation}
Since there are two Lagrange multipliers and
$
\widetilde{B}(i)>0,
$
from (\ref{jizhidianyong}) and (\ref{jizhidianyong2}), there exist
$
r>l>k
$
such that
$
|a_{ir}|\cdot|a_{il}|\neq0,\;\ld_{r}>\ld_{l},
$
and
$
|a_{ij}|=0,\;j\neq r,l.
$
Hence, we have
\begin{equation}\label{haole}
\left\{\begin{aligned}
&m_{r}(\ld_{r}-\ld_{i})^2|a_{ir}|^2+m_{l}(\ld_{l}-\ld_{i})^2|a_{il}|^2=\widetilde{A}(i),\\
&m_{r}|a_{ir}|^2+ m_{l}|a_{il}|^2=\widetilde{B}(i),
\end{aligned}\right.
\end{equation}
where $m_{r}, m_{l}$ are the multiplicity of the eigenvalues $\ld_{r}\,\mbox{and}\,\ld_{l},$
respectively.
From
(\ref{haole}),
we have
\begin{equation}\label{Equ-C}
\widetilde{C}(i)=\frac{\widetilde{A}(i)+(\ld_{r}-\ld_{i})(\ld_{l}-\ld_{i})\widetilde{B}(i)}{(\ld_{l}-\ld_{i})+(\ld_{r}-\ld_{i})}.
\end{equation}
From
(\ref{fuschwarzinequality}),
we have
\begin{equation}\label{tiaojian1}
\widetilde{C}(i)\leq\sqrt{\widetilde{A}(i)\widetilde{B}(i)}.
\end{equation}
By the definition of
$
\widetilde{A}(i)
$
and
$
\widetilde{B}(i),
$
we have
\begin{equation}\label{tiaojian2}
(\ld_{k+1}-\ld_{i})\leq \sqrt{\widetilde{A}(i)/\widetilde{B}(i)}.
\end{equation}
From the range of the function
$
\frac{\widetilde{A}(i)+(\ld_{r}-\ld_{i})(\ld_{l}-\ld_{i})\widetilde{B}(i)}{(\ld_{l}-\ld_{i})+(\ld_{r}-\ld_{i})},
$
(\ref{tiaojian1}),
we have $r=k+2.$
From
(\ref{tiaojian2}),
we have
$
l=k+1.
$
Therefore, we obtain
\begin{equation}\label{chubujieguo}
\widetilde{C}(i)\leq \frac{\widetilde{A}(i)+(\ld_{k+2}-\ld_{i})(\ld_{k+1}-\ld_{i})\widetilde{B}(i)}{(\ld_{k+2}-\ld_{i})+(\ld_{k+1}-\ld_{i})}.
\end{equation}
From
(\ref{chubujieguo}),
and the definition of
$\widetilde{A}(i),\;\widetilde{B}(i)$ and $\widetilde{C}(i),$
we have
\begin{equation}\label{chubu2}
\aligned
 &\left((\ld_{k+2}-\ld_{i})+(\ld_{k+1}-\ld_{i})\right)\into |\nabla g|^2u_{i}^2\\
\leq&\into\left|2\nabla g\cdot\nabla u_{i}+ u_{i}\Delta g \right|^2
+(\ld_{k+1}-\ld_{i})(\ld_{k+2}-\ld_{i})\into |gu_{i}|^2\\&-\sum_{j=1}^{k}(\ld_{k+1}-\ld_{j})(\ld_{k+2}-\ld_{j})|a_{ij}|^{2}\\
\leq&\into\left|2\nabla g\cdot\nabla u_{i}+ u_{i}\Delta g \right|^2
+(\ld_{k+1}-\ld_{i})(\ld_{k+2}-\ld_{i})\into|gu_{i}|^2
\endaligned
\end{equation}
which finishes the proof of Lemma \ref{gapkeylemma}.
\end{proof}

Based on Lemma \ref{gapkeylemma}, we have
\begin{cor}\label{yongzhegehao}
Under the assumption of Lemma \ref{gapkeylemma}, for any real value function $f\in C^{3}(\Omega)\cap C^{2}(\overline{\Omega}),$  we have
\be\label{zhegehao}
\ba
    &\left((\ld_{k+2}-\ld_{i})+(\ld_{k+1}-\ld_{i})\right)\into |\n f|^{2}u_{i}^{2}\\
\leq&2\sqrt{\left((\ld_{k+2}-\ld_{i})(\ld_{k+1}-\ld_{i})\right)\into |\n f|^{4}u^{2}}+\into \left(2\n f\cdot \n u_{i}+u_{i}\Delta f\right)^{2}.
\ea
\ee
\end{cor}
\begin{proof}
Taking $g=\exp(\sqrt{-1}\al f),\,\al\in\Bbb{R}\backslash \{0\}$ in (\ref{mainformula}), we have
\begin{equation}\label{zhegehao1}
\aligned
&\al^{2}\left((\ld_{k+1}-\ld_{i})+(\ld_{k+2}-\ld_{i})\right)\into |\nabla f|^2u_{i}^2\\
\leq &\al^{4}\into |\n f|^{4}u_{i}^{2}+\al^{2}\into\left|2\nabla f\cdot\nabla u_{i}+ u_{i}\Delta f \right|^2
+(\ld_{k+1}-\ld_{i})(\ld_{k+2}-\ld_{i}).
\endaligned
\end{equation}
From (\ref{zhegehao1}), we have
\begin{equation}\label{zhegehao2}
\aligned
&\left((\ld_{k+1}-\ld_{i})+(\ld_{k+2}-\ld_{i})\right)\into |\nabla f|^2u_{i}^2\\
\leq &\al^{2}\into |\n f|^{4}u_{i}^{2}
+\frac{1}{\al^{2}}(\ld_{k+1}-\ld_{i})(\ld_{k+2}-\ld_{i})+\into\left|2\nabla f\cdot\nabla u_{i}+ u_{i}\Delta f \right|^2.
\endaligned
\end{equation}
Using the Cauchy-Schwarz inequality in (\ref{zhegehao2}), we have
(\ref{zhegehao}).
\end{proof}

\begin{cor}\label{KeyLem3} Under the assumption of Lemma \ref{gapkeylemma}, for any real value function $f\in C^{3}(\Omega)\cap C^{2}(\overline{\Omega})$ satisfying  $|\n f|^2=1$,  we have
\begin{equation}\label{Key-Inq1}
 (\ld_{k+2}-\ld_{k+1})^2\leq 16\left(\into(\n f\cdot\n u_i)^2-\frac14\into (\Delta f)^2u_i^2-\frac12\into(\n(\Delta f)\cdot\n f) u_i^2\right)
\ld_{k+2}.
\end{equation}
Furthermore, we have
\begin{equation}\label{Key-Inq2}
\ld_{k+2}-\ld_{k+1}\leq 4\left(\ld_i-\frac14\into (\Delta f)^2u_i^2-\frac12\into(\n(\Delta f)\cdot\n f) u_i^2\right)^{\frac12}\sqrt{\ld_{k+2}}.
\end{equation}
 \end{cor}
\begin{proof} From Corollary \ref{yongzhegehao} and $|\n f|^2=1$, we have
\begin{equation*}
 \left((\ld_{k+2}-\ld_{i})+(\ld_{k+1}-\ld_{i})\right)-2\sqrt{(\ld_{k+2}-\ld_{i})(\ld_{k+1}-\ld_{i})}\leq\into
 \left(2\n f\cdot \n u_{i}+u_{i}\Delta f\right)^{2},
\end{equation*}
i.e.
\begin{equation*}
\left(\sqrt{\ld_{k+2}-\ld_{i}}-\sqrt{\ld_{k+1}-\ld_{i}}\right)^2\leq \into \left(2\n f\cdot \n u_{i}+u_{i}\Delta f\right)^{2}.
\end{equation*}
By integration by parts, we have
\begin{equation*}
\into \left(2\n f\cdot \n u_{i}+u_{i}\Delta f\right)^{2}=4\into(\n f\cdot\n u_i)^2-\into (\Delta f)^2u_i^2-2\into (\n(\Delta f)\cdot\n f) u_i^2.
\end{equation*}
Hence, we have
\begin{equation}\label{K1}
\left(\sqrt{\ld_{k+2}-\ld_{i}}-\sqrt{\ld_{k+1}-\ld_{i}}\right)^2\leq
4\into(\n f\cdot\n u_i)^2-\into (\Delta f)^2u_i^2-2\into(\n(\Delta
f)\cdot\n f) u_i^2
\end{equation}
Multiplying (\ref{K1}) by $\left(\sqrt{\ld_{k+2}-\ld_{i}}+\sqrt{\ld_{k+1}-\ld_{i}}\right)^2 $ on both sides, we can get
\begin{equation*}
\ba
(\ld_{k+2}-\ld_{k+1})^2\leq & 4\left(\into(\n f\cdot\n u_i)^2-\frac14\into (\Delta f)^2u_i^2-\frac12\into(\n(\Delta f)\cdot\n f) u_i^2\right)\\
&\times\left(\sqrt{\ld_{k+2}-\ld_{i}}+\sqrt{\ld_{k+1}-\ld_{i}}\right)^2\\
\leq& 16 \left(\into(\n f\cdot\n u_i)^2-\frac14\into (\Delta
f)^2u_i^2-\frac12\into(\n(\Delta f)\cdot\n f) u_i^2\right)\ld_{k+2},
\ea
\end{equation*}
which is  the inequality (\ref{Key-Inq1}).

From the Cauchy-Schwarz inequality and integration by parts, we
obtain
\begin{equation*}
(\ld_{k+2}-\ld_{k+1})^2\leq 16 \left(\ld_i-\frac14\into (\Delta
f)^2u_i^2-\frac12\into(\n(\Delta f)\cdot\n f) u_i^2\right)\ld_{k+2}.
\end{equation*}
Finally we have (\ref{Key-Inq2}).

\end{proof}

\begin{proof}[Proof of Theorem \ref{MainThm1} ]
Let
$
x_{1},\;x_{2},\cdots,x_{n}
$
be the standard coordinate functions in $\Bbb R^n$. Since $|\n x_l|=1, l=1,\cdots,n$, we can use Lemma \ref{KeyLem3}.
Taking
$$
f=x_{l},\,l=1,\cdots,n,\,\mbox{and}\,\, i=1
$$
in (\ref{Key-Inq1}) and then taking sum over $l$ from $1$ to $n$,   we have
\begin{equation}\label{whyeuli}
 \ba n(\ld_{k+2}-\ld_{k+1})^2\leq &16\ld_{k+2}\into\sum_{l=1}^n\left(\frac{\partial u_1}{\partial x_l}\right)^2\\
 =& 16\ld_1 \ld_{k+2}       .
\ea\end{equation}
From Theorem 3.1 in \cite{CY4} (see also (\ref{diffyong})), from (\ref{whyeuli}), we deduce
\begin{equation*}
\ba\ld_{k+2}-\ld_{k+1}\leq &4\sqrt{\frac{\ld_1}{n}}\sqrt{\ld_{k+2} }\\
\leq & 4\ld_1\sqrt{\frac{C_0(n)}{n}}(k+1)^{\frac12}\\
=&C_{n,\Omega}(k+1)^{\frac12},
\ea\end{equation*}
where $C_{n,\Omega}= 4\ld_1\sqrt{\frac{C_{0}(n)}{n}}$, $C_{0}(n)$ is given by $(\ref{diffyong})$.
Since $k$ is arbitrary, this completes the proof of Theorem \ref{MainThm1}.
\end{proof}

Although Corollary \ref{gapexample1} can be deduced directly by taking $a=b=1$ in Corollary \ref{MainThm2}, its proof is  interesting independently.  Here we give its proof for the upper half-plane model of hyperbolic space.
\begin{proof}[Proof of Corollary \ref{gapexample1}]
For convenience, we will use the upper half-plane model of the hyperbolic space, that is,
$$
\mathbb H^{n}(-1)=\{(x_{1},\cdots,x_{n})\in \Bbb{R}^n|x_{n}>0\}
$$
with the standard metric
$$
\mathrm{d}s^2=\frac{(\mathrm{d}x_{1})^2+\cdots+(\mathrm{d}x_{n})^2}{(x_{n})^2}.
$$
Taking
$
r=\log x_{n},
$
we have
$$
\mathrm{d}s^2=(\mathrm{d}r)^2+\mathrm{e}^{-2r}\sum_{i=1}^{n-1}(\mathrm{d}x_{i})^2.
$$
Since $|\n r|=1,\;\Delta r=-(n-1)$, taking
$
f=r\,\,\mbox{and}\, i=1
$
in (\ref{Key-Inq2}), we have
\begin{equation}\label{whyhyp}
\ba\ld_{k+2}-\ld_{k+1}\leq &4\left(\ld_1-\frac14\into (\Delta r)^2u_i^2-\frac12\into(\n(\Delta r)\cdot\n r) u_1^2\right)^{\frac12}\sqrt{\ld_{k+2}}\\
=&4\left(\ld_1-\frac{(n-1)^{2}}{4}\right)^{\frac12}\sqrt{\ld_{k+2}}.
\ea\end{equation}
By the result in
\cite{ChenCheng} (see also (\ref{chenchenganhan})), from (\ref{whyhyp}), we have
\begin{equation*}
\ba\ld_{k+2}-\ld_{k+1}\leq&  4\left(\ld_1-\frac{(n-1)^{2}}{4}\right)^{\frac12}\sqrt{C_0(n)\left(\ld_{1}
+\frac{n^2}{4}H_{0}^2\right)}(k+1)^{\frac1n}\\
=&C_{n,\Omega}(k+1)^{\frac1n},
\ea\end{equation*}
where $C_{n,\Omega}$ is defined by (\ref{Const1}).
Since this inequality holds for any $k$,  we can deduce (\ref{U-inequ3}).
\end{proof}

\section{Proof of Corollary \ref{MainThm2} }
\begin{proof}[Proof of Corollary  \ref{MainThm2}]
Assume that $\Omega$ is a bounded domain in an $n$-dimensional complete  noncompact Riemannian  manifold $(M,\,g)$ with sectional curvature $Sec$ satisfying $-a^2\leq Sec\leq -b^2$, where $0 \leq b\leq a $ are constants.  For $p \notin  \overline{\Omega}$ fixed, define the distance function by
$\rho(x)=\mbox{distance}(x,\,p)$.
From Proposition 2.2 in P.15 of \cite{schoenyau}, and $|\n \rho|=1,$  we have
\begin{equation}\label{3191}
\n \rho\cdot\n( \Delta \rho)=-|\text{\upshape Hess}\ \rho|^2-\text{\upshape Ric}( \n\rho, \n\rho).
\end{equation}
Let $0\leq h\leq h_1,\cdots, h_{n-1}\leq H$ be the eigenvalues of the $\mr{Hess}\rho$. Then we have
\be\label{youyongzhiyuma}
\ba
&2|\mbox{Hess}\rho|^2-(\Delta \rho)^2\\
=&2\sum_{i=1}^{n-1}h_{i}^2-\lf(\sum_{i=1}^{n-1}h_{i}\rt)^2\\
=&\sum_{i=1}^{n-1}h_{i}^2-\sum_{i\neq j}h_{i}h_{j}\\
\leq& h_{n-1}^2+h_{1}h_{2}+\cdots+h_{n-2}h_{n-1}-\sum_{i\neq j}h_{i}h_{j}\\
=&h_{n-1}^2-h_{1}h_{2}-\cdots-h_{n-2}h_{n-1}-\sum_{i\neq j\atop i,j\leq n-2}h_{i}h_{j}\\
\leq&H^2-(n-2)^2h^2.
\ea
\ee
From the Hessian comparison theorem (cf. \cite{wuhongxi}), under the conditions in Corollary \ref{MainThm2}, we have
\be\label{hesseigen}
a\frac{\cosh a\rho}{\sinh a \rho}\geq h_{n-1}\geq\cdots\geq h_1\geq b\frac{\cosh b\rho}{\sinh b \rho}.
\ee
Since $n\geq 3$ and $\frac{a^2}{\sinh^2a\rho}$ is a decreasing function of $a$, from (\ref{youyongzhiyuma}) and (\ref{hesseigen}), under the conditions in Corollary \ref{MainThm2}, we have
\be\label{youyong}
\ba
&2|\mbox{Hess}\rho|^2+2\mbox{Ric}(\n\rho,\,\n\rho)-(\Delta \rho)^2\\
\leq& a^2\frac{\cosh^2a\rho}{\sinh^2a\rho}-(n-2)^2b^2\frac{\cosh^2b\rho}{\sinh^2b\rho}-2(n-1)b^2\\
=&a^2+\frac{a^2}{\sinh^2a\rho}-(n-2)^2b^2-(n-2)^2\frac{b^2}{\sinh^2b\rho}-2(n-1)b^2\\
\leq&-(n-1)^2b^2+(a^2-b^2)+\frac{b^2}{\sinh^2b\rho}-(n-2)^2\frac{b^2}{\sinh^2b\rho}\\
\leq&-(n-1)^2b^2+(a^2-b^2).
\ea
\ee

Taking
$
f=\rho\,\,\mbox{and}\, i=1
$
in (\ref{Key-Inq2}), we have
\begin{equation}\label{why1}
\ld_{k+2}-\ld_{k+1}\leq 4\left(\ld_1-\frac14\into (\Delta \rho)^2u_1^2-\frac12\into(\n( \Delta \rho)\cdot\n \rho) u_1^2\right)^{\frac12}\sqrt{\ld_{k+2}}.
\end{equation}
From (\ref{3191}) and (\ref{youyong}), we obtain
\begin{equation}\label{why2}
\ba
&\ld_1-\frac14\into (\Delta \rho)^2u_1^2-\frac12\into(\n( \Delta \rho)\cdot\n \rho) u_1^2\\
=&\ld_1+\frac14\into\lf(2|\mr{Hess}\rho|^2+2\mr{Ric}(\n\rho,\,\n\rho)-(\Delta \rho)^2\rt)u_{1}^2\\
\leq &\ld_{1}-\frac{(n-1)^2}4b^2+\frac{a^2-b^2}{4}.
\ea\end{equation} By the result in \cite{ChenCheng} (see also
(\ref{chenchenganhan})), from (\ref{why1}) and (\ref{why2}), we have
\begin{equation*}
 \ba\ld_{k+2}-\ld_{k+1}\leq&4\left(\ld_{1}-\frac{(n-1)^2}4b^2+\frac{a^2-b^2}{4}\right)^{\frac12}
 \sqrt{C_{0}(n)\left(\ld_{1}+\frac{n^2}{4}H_{0}^2\right)} (k+1)^{\frac{1}{n}}\\
 \leq&C_{n,\Omega}(k+1)^{\frac{1}{n}},
\ea\end{equation*}
where $C_{n,\Omega}$ is defined by (\ref{Const2}). Since this inequality holds for any $k$,  we can deduce (\ref{U-inequ2}).
\end{proof}

\begin{flushleft}
Daguang Chen\\
Department of Mathematical Sciences, Tsinghua University, Beijing, 100084, P. R. China \\
E-mail: dgchen@math.tsinghua.edu.cn
\end{flushleft}

\begin{flushleft}
Tao Zheng\\
Department of Mathematics, Beijing Institute of Technology, Beijing 100081, P. R. China \\
E-mail: zhengtao08@amss.ac.cn
\end{flushleft}

\begin{flushleft}
Hongcang Yang\\
Hua Loo-Keng Key Laboratory of Mathematics, Chinese Academy of Sciences, Beijing 100080, P. R. China\\
E-mail:yanghc2@netease.com
\end{flushleft}

\end{document}